\documentclass{amsart}

\newtheorem{thm}{Theorem}[subsection]

\newtheorem{lem}[thm]{Lemma}
\newtheorem{prop}[thm]{Proposition}
\theoremstyle{definition}
\newtheorem{defn}[thm]{Definition}
\theoremstyle{remark}
\newtheorem{rem}[thm]{Remark}
\numberwithin{equation}{subsection}

\newcommand{\eps}{\varepsilon}
\newcommand{\h}{\mathcal{H}}
\newcommand{\C}{\mathcal{C}}
\newcommand{\F}{\mathcal{F}}
\newcommand{\G}{\mathcal{G}}
\newcommand{\El}{\mathcal{L}}

\newcommand{\W}{\mathcal{W}}
\newcommand{\U}{\mathcal{U}}
\newcommand{\PH}{{\textsf{PH}}}
\newcommand{\Real}{\mathbb{R}}
\newcommand{\Nat}{\mathbb{N}}
\newcommand{\Int}{\mathbb{Z}}
\newcommand{\Toro}{\mathbb{T}}
\newcommand{\Vol}{\upsilon o\ell}
\newcommand{\ve}{\upsilon}

\newcommand{\abs}[1]{\left\vert#1\right\vert}
\newcommand{\set}[1]{\left\{#1\right\}}
\newcommand{\seq}[1]{\left<#1\right>}
\newcommand{\norm}[1]{\left\Vert#1\right\Vert}
\newcommand{\Diff}{\operatorname{Diff}}

\newtheorem*{thmx}{Theorem}
\newtheorem*{thma}{{\bf Theorem A}}
\newtheorem*{thmb}{{\bf Theorem B}}
\newtheorem*{thmc}{{\bf Theorem C}}

\begin{document}

\title[Hyperbolic Sets and Homological Entropy]
 {Hyperbolic Sets and Entropy\\ at the Homological Level}

\author{Mario Roldan}

\address{IMPA, Instituto Nacional de Matem\'{a}tica Pura e Aplicada, Estrada Dona Castorina 110, Jardim Botanico, Rio de Janeiro 22460--320, Brasil}

\email{roldan@impa.br}

\thanks{{\em Date}. July, 2014}
\thanks{2010 {\em Mathematics Subject Classification}. Primary 37D25; Secondary 37B40, 37D30}
\thanks{{\em Key words and phrases}. topological entropy, homological entropy, partial hyperbolicity}
\thanks{The author would like to express deep gratitude to {\em Enrique Pujals} and {\em Rafael Potrie} whose guidance and support were crucial for the successful completion of this project. The author also would like to thank to Mart\'{\i}n Sambarino and Pablo Carrasco for useful conversations. The author also would like to thank IMPA-Brazil, IMERL, CMAT-Montevideo  for their warm hospitality, institutions where the results of this paper were obtained during the Ph.D. studies. This work was completed with the support of CNPq-Brazil.}





\begin{abstract}
  The aim of this work is to study a kind of refinement of the entropy conjecture, in the context of partially hyperbolic diffeomorphism with one dimensional central direction, of $d$-dimensional torus. We start by establishing a connection between the unstable index of hyperbolic sets and the index at algebraic level. Two examples are given which might shed light on which are the good questions in the higher dimensional center case.
\end{abstract}

\maketitle

\section*{Introduction}

Let us consider $f\colon M\to M$ a diffeomorphism of a compact manifold. Concerning smooth maps and their action at an algebraic level, the well known property, which may hold for all dynamical systems, is the so called {\em entropy conjecture}~\cite{Shub74}. The precise statement goes as follows:

\medskip
{\bf Entropy Conjecture}: For any $C^{1}$-map $f\colon M\to M$ of a compact manifold $M$, $h_{top}(f)$ is bounded from below by $h_H(f)$. \medskip

Here $h_{top}(f)$ denotes the {\em topological entropy} of a continuous map $f$ and $h_H(f)$ denotes the {\em homological entropy} of $f$, terms which are explained below.

The major progress was due to Yomdin~\cite{Yom87}, where the entropy conjecture was proved for any $C^{\infty}$ diffeomorphisms.

\medskip

At this point we consider the following question which is analyzed in this work:

\medskip

{\bf Basic Question:} For a $C^{1}$ diffeomorphism $f\colon M\to M$, consider $u_{0}\in \Nat$ such that $h_H(f)=\log sp(f_{\ast, u_0})$ and assume that $f_{\ast, 1}$ is hyperbolic. When, or under which condition
does there exist uniformly hyperbolic sets $\Lambda_\ell\subset M$ with hyperbolic splitting of the tangent bundle, $T_{\Lambda_{\ell}}M = E^{u}_{\ell}\oplus E^{s}_{\ell}$, such that $\dim E^{u}_{\ell}=u_{0}$ and
$$\limsup_{\ell\to\infty} \ h_{top}(f\mid \Lambda_{\ell})\geq h_H(f)?$$

The number $u_{0}\in \Nat$ such that $h_H(f)=\log sp(f_{\ast, u_0})$ is called {\em algebraic index} or {\em homological index}. If $\Lambda$ is a hyperbolic set, the dimension of its unstable bundle is called {\em unstable index} of $\Lambda$.

\medskip

Recall that, the topological entropy has to do with the exponential growth rate of $u$-dimensional unstable disks, while the entropy in homology refers with how the dynamics disfigures the manifold at an algebraic level, and the quantity $h_H(f)$ measures how much and how many times the dynamic mixes up the $u$-dimensional cells. One can thus expect that hyperbolic sets with unstable index equal to the homological index will exist and hopefully contribute at least as much entropy as $h_H(f)$.

We try with this refinement to understand the structure of invariant sets that contributes to the growth of topological entropy. The prevailing idea for many systems is that hyperbolic sets are the heart of dynamic complexity.

Notice that, the answer to the basic question is in general negative if the action induced on the first homology group is not hyperbolic. For example it is possible to have a diffeomorphism with positive entropy, which simultaneously lacks hyperbolic sets. This can be done even in the partially hyperbolic cases, look for example at, $Anosov\times Id$ in $\Toro^{3}$.

The lower bound proposed by the entropy conjecture is clearly not an upper bound since it is easy to construct diffeomorphisms in every isotopy class with an arbitrarily large entropy. On the other hand, some families of maps may verify that the entropy in homology is also an upper bound: to determine if a family of maps satisfies such a bound is another interesting question which will be addressed in this work.

This paper considers continuous maps defined on the $d$-torus, where is well known that the entropy conjecture holds~\cite{MisPrzy77}, and culminates in an affirmative answer in the case of partially hyperbolic system with one-dimensional center bundle, isotopic to an linear Anosov, when the isotopy is assumed to be in the set of partially hyperbolic diffeomorphisms. Some examples are presented where the upper bound is not attained.

When dealing with the existence of hyperbolic sets it is a common practice to study the existence of hyperbolic measures. For the case of partially hyperbolic system, it means to study the sign of the central Lyapunov exponents~\cite{Ures12}. We follow this approach and establish a connection between the unstable index and algebraic index.

Examples in~\cite{HHTU10} shows that the answer to the basic question is also possible when the algebraic index is not unique. Furthermore, there seems to be evidence, of what should be an expected answer, in a dense and open set.

\section{Presentation of Results}

This section contains a brief summary of the results obtained. The terminology involved in the theorems are explained in the next section.

\subsection{Precise Setting}

 Before stating the precise results it is essential to mention a result by Fisher, Potrie, and Sambarino~\cite{FPS13}. In the study of dynamical coherence for partially hyperbolic diffeomorphism of torus, they have established the existence of a unique maximal entropy measure in the case where a partially hyperbolic diffeomorphism can be connected to a hyperbolic linear automorphism via a path which remains within the set of partially hyperbolic diffeomorphisms. For the existence and uniqueness of the measure, the central bundle is required to be one-dimensional.

\medskip

We start by briefly introducing the definition of partial hyperbolicity.

\begin{defn}
To say that $f\colon M\to M$ is a {\em partially hyperbolic} (pointwise) diffeomorphism we need three conditions:
\begin{enumerate}
  \item {\em Splitting condition}: There exists a continuous splitting of the tangent bundle $TM = E^{u}\oplus E^{c} \oplus E^{s}$, which is $Df$-invariant i.e., $DfE^{\sigma}_{x}=E^{\sigma}_{fx}$ \, for $\sigma=u, c, s.$
  \item {\em Domination condition}: There exists  $N>0$ such that for every $x\in M$ and unit vectors $v^{\sigma}\in E^{\sigma}$ ($\sigma = u,c,s$) we have \[ \norm{D_{x}f^{N}v^{s}} < \norm{D_{x}f^{N}v^{c}} < \norm{D_{x}f^{N}v^{u}}.\]
  \item {\em Contraction property}: $\norm{Df^{N}|_{E^{s}}}<1$ and $\norm{Df^{-N}|_{E^{u}}}<1$.
\end{enumerate}
  If the nature of the domination above holds for unit vectors belonging to the bundles of different points, then we say that $f$ is {\em absolutely} partially hyperbolic, i.e., \[\norm{D_xf^{N}v^{s}} < \norm{D_{y}f^{N}v^{c}} < \norm{D_{z}f^{N}v^{u}} \] for all $x,y,z\in M$ and $v^{\sigma}\in E^{\sigma}_{p}$ unit vectors ($\sigma = u, c, s$ and $p = x, y, z $), respectively.
\end{defn}

\begin{rem}
    For partially hyperbolic diffeomorphisms, it is a well-known fact that there are foliations $\W^{u}$ (unstable one) and $\W^{s}$ (stable one) tangent to the distributions $E^{u}$, $E^{s}$ respectively. Moreover, the bundles $E^{u}$ and $E^{s}$ are uniquely integrable~\cite{HPS77}. If there exist invariant foliations $\W^{c\sigma}$ tangent to $E^{c\sigma}=E^c\oplus E^\sigma$ for $\sigma=s,u$ then $f$ is said to be {\em dynamically coherent}.
\end{rem}

\begin{defn}
Let $\mu$ be a hyperbolic ergodic measure and a number $u \in\Nat$. It is
said that $\mu$ has {\em index} equal to $u$ if there exist $u$ positive
Lyapunov exponents. Here the exponents are counted with multiplicity.
\end{defn}


Let us start by establishing the set where the results will be held. First, the set of all partially hyperbolic diffeomorphisms of the $d$-dimensional torus are denoted by $\PH(\Toro^{d}).$
We consider $A\in SL(d,\Int)$, a $d\times d$ matrix with integer entries and determinant 1. Let $T\Toro^{d} = E^{u}_{A}\oplus E^{c}_{A}\oplus E^{s}_{A}$ be the dominated splitting associated to the induced diffeomorphism $f_A$ (which is also denoted by $A$) of the quotient $\Real^d/\Int^d=\Toro^d$.

Recall that for every continuous map $f$ on $\Toro^d$ there exists a unique linear map $A_f\colon \Real^{d}\to \Real^{d}$ such that the lift of $f$ to the universal covering map $\widetilde{f}$ is homotopic to $A_f$. We call $A_f$ the {\em linear part} of $f$.

We will now consider the subset of all partially hyperbolic diffeomorphisms, all of them having the same dimension of stable and unstable bundle, i.e.
\[\PH_{A,u,s} = \set{f\in \PH\,\,|\, f\sim A, \, \dim E^{u}_{f}= u, \,\,\dim E^{s}_{f}=s}.\]
Here $f\sim g$ denotes an isotopy between $f$ and $g$. To simplify notation, will be $\PH_{A}(\Toro^{d})$ written instead of $\PH_{A,u,s}(\Toro^{d})$, where the dimension of the bundles is implicitly understood. Now, we consider $\PH^{0}_{A}(\Toro^{d})$ to be the connected component of $\PH_{A}(\Toro^{d})$ containing the lineal map $A$.

\medskip

Then, in~\cite{FPS13} it is proved that:
\begin{thm}[Fisher, Potrie, Sambarino]
\label{t:FPS}
   For every $f\in$ {\em $\PH^{0}_{A}$}$(\Toro^{d})$ such that $\dim E^{c}_{f}=1$ there exists a unique maximal entropy measure which has entropy equal to the linear part.
\end{thm}

\subsection{Statement of Results}

 In the context presented above, the first refinement of the entropy conjecture is performed over $\PH^{0}_{A}(\Toro^{d})$, considering those who have one-dimensional central direction. However, the result is valid on any connected component containing some diffeomorphism which admits a special closed form.

\begin{thma} \em
    Let $f\in$ {\em $\PH^{0}_{A}$}$(\Toro^{d})$ with $\dim E^{c}_{f}=1$. Then the following holds:
    \begin{enumerate}
      \item Let $\mu$ be an ergodic $f$-invariant measure, $0<\eps_0<\abs{\lambda^{c}_A}$. Suppose that $h_{\mu}(f)>h(A)-\eps_0$ then $\mu$ is a hyperbolic measure with index $u=\dim E^{u}_A$.
      \item For every $\eps>0$ there exists a hyperbolic set $\Lambda_{\eps}\subset\Toro^{d}$ such that
          $$h_{top}(f\mid\Lambda_{\eps})\geq \log sp(f_{\ast,u})-\eps$$ where $u=\dim E^{u}_{\Lambda_{\eps}}$, and $h_H(f)=\log sp(f_{\ast,u})$.
    \end{enumerate}
\end{thma}

\begin{rem}
    In particular, the unique maximal entropy measure given by Theorem~\ref{t:FPS} is actually a hyperbolic measure.
\end{rem}

We wonder what happens if the central direction is no longer one-dimensional. In this direction, we present two examples (absolutely and pointwise) of partially hyperbolic diffeomorphisms on $\Toro^{4}$, both of them with two dimensional center direction. Among other things, both examples show us that the homological entropy is not an {\em upper bound} for the topological entropy. The first example ensures hyperbolic measure with good index, which is sufficient to approximate the homological entropy, i.e.,

\begin{thmb}\label{t:B} \em
    There exists an open subset $\U\subset \Diff^{1}(\Toro^{4})$ such that any $f\in \U$ is an {\em absolutely} partially hyperbolic diffeomorphism with $\dim E^{c}_{f}=2$. Furthermore,
    \begin{enumerate}
     \item $f$ is homotopic to a hyperbolic linear map $A$ that has unstable index $2$.
     \item $f$ satisfies $h_{top}(f)\geq h_{top}(f\mid W^{c}) > h_H(f)=\log sp(A_{\ast,2})$.
     \item Let $\mu$ be an ergodic measure such that $h_{\mu}(f)>\log sp(A_{\ast,1})$ then $\mu$ is a hyperbolic measure with index $2$.
      \item $f$ is topologically transitive.
    \end{enumerate}
\end{thmb}

 In the second example, although hyperbolic measure is guaranteed with entropy as close as desired to the topological entropy, its index is no longer correct. We believe that there exists a hyperbolic measure with good index in order to approximate the homological entropy.

\begin{thmc} \label{t:C} \em
    There exists an open subset $\U\subset \Diff^{1}(\Toro^{4})$ such that any $f\in \U$ is a {\em pointwise} partially hyperbolic diffeomorphism with $\dim E^{c}_{f}=2$. Furthermore,
    \begin{enumerate}
      \item $f$ is homotopic to a hyperbolic linear map $A$ that has unstable index $1$.
      \item $f$ satisfies $h_{top}(f)\geq h_{top}(f\mid W^{c})> h_H(f)=\log sp(A_{\ast,1})$.
      \item Any ergodic measure $\mu$, such that $h_{\mu}(f)>\log sp(A_{\ast,1})$ is a hyperbolic measure with index $2$. If $\Lambda\subset\Toro^{4}$ is a hyperbolic set with index $1$ then           $h_{top}(f\mid\Lambda)\leq h_{top}(A)$.
      \item $f$ is topologically transitive.
    \end{enumerate}
\end{thmc}

In the theorems above, $W^{c}$ denotes some periodic central leaf.

\section{Definitions and Notations}

In this section, definitions and terminology related to the complexity of a system, normal hyperbolicity and foliation tools, are introduced. We recommend the books~\cite{KatHass, Walters} for a better understanding and a source of background reading.

\subsection{Exponential Volume Growth}
Given a sequence of numbers $(a_n)\subset (0,\infty]$, the following limit defined by
\[\tau(a_n)=\limsup_{n\to\infty}\frac{1}{n}\log a_n\] is used to denote the {\em exponential growth rate}. The numbers (topological entropy, metric entropy, entropy at homology level, volume growth of the foliation), that will appear in the definitions below, are the exponential growth rate of a particular given sequence.

\medskip

{\em Topological Entropy}: Consider a continuous map, $f\colon M\to M $, of a compact metric space $M$, with distance function $d$. A set $E\subset M$ is said to be $(n,\eps)$-{\em separated} if for every $x\neq y \in E$ there exists $i\in \{0,\dots,n-1\}$ such that $d(f^{i}x,f^{i}y)\geq\eps$. Let $s(n,\eps)$ be the maximal cardinality of an $(n,\eps)$-separated set of $M$, notice that by compactness, this number is finite. One defines
\[h(f,\eps)=\limsup_{n\to\infty}\frac{1}{n}\log s(n,\eps),\] which represents the exponential growth rate of a sequence $a_n=s(n,\eps)$. The topological entropy, as defined by Bowen, is the ``greatest'' of all
those growth rates, \[h_{top}(f)=\lim_{\eps\to 0} h(f,\eps)=\sup_{\eps>0} h(f,\eps).\]

\medskip

{\em Metric Entropy}: We recall the definition of metric entropy as defined by B. Katok, with respect to the Borel probability $f$-invariant ergodic measure, $\mu$. In the literature this number is said to be one which measures the
complexity of the orbits of a system that are relevant for a measure $\mu$.

For $0<\delta<1$, $n\in\Nat$ and $\eps>0$, a finite set of $E\subset M$ is called an $(n,\eps,\delta)$-{\em covering} set if the union of the all $\eps$-balls, $B_{n}(x,\eps)= \set{y\in \, M | \,\ d(f^{i}x,f^{i}y)<\eps}$, centered at points $x\in E$  has $\mu$-measure greater than $\delta$. Subsequently we consider the set $N_{\delta}(n,\eps,\delta)$ as being the smallest possible cardinality of a $(n,\eps,\delta)$-covering set (i.e. $N_{\delta}(n,\eps,\delta)=\min\set{\sharp E\ | \ \ E \ \ \text{is
a} \ \ (n,\eps,\delta)-\text{covering set}})$. Once again, the metric entropy of $f$ with respect to measure $\mu$ is an exponential growth rate of a sequence $a_n=N_{\delta}(n,\eps,\delta)$, i.e., \[h_{\mu}(f)=\lim_{\eps\to 0}\limsup_{n\to\infty}\frac{1}{n}\log N_{\delta}(n,\eps,\delta).\]

A very significant result in the theory relating the metric entropy and the topological entropy is the so called {\em entropy variational principle}~\cite{Man83}, which states that the topological entropy is achieved by taking the supremum of the metric entropies of all invariant measures of the systems: \[h_{top}(f)=\sup_{\mu}h_{\mu}(f). \]
If there is $\mu$ such that $h_{top}(f)=h_{\mu}(f)$, then $\mu$ is said to be a {\em maximal} entropy measure.

\medskip

{\em Homological Entropy}: For a continuous map $f\colon M\to M$ consider the sequence of homomorphisms $f_{\ast,k}\colon H_k(M,\Real)\to H_k(M,\Real)$, on the homology groups with real coefficients, where $k=0,\dots,d=\dim M.$
The homological entropy, denoted by $h_H(f)$, is defined as follows, \[h_H(f)=\max_{k} \log
sp(f_{\ast,k}).\]

Notice that this real number maximizes the above over all $k\in\{0,\cdots,d\}$, and it is the exponential growth rate of the sequence $a_n=\|f^{n}_{\ast,k}\|$. The well known property in the theory is that the action induced in homology is invariant under homotopy.

\medskip

{\em Volume Growth of the Foliation}: Consider $f\colon M\to M$ a diffeomorphism on a $d$-dimensional compact Riemannian manifold, $M$. Let us consider $\W=\{W(x)\}_{M\ni x}$ a $u$-dimensional foliation on $M$ which is invariant under $f$, i.e. $fW(x)=W(fx)$. For any $x\in M$, let $W_r(x)$ be the $u$-dimensional disk on $W(x)$ centered at $x$, with radius $r>0$.

Let $\Vol(W)$ denote the {\em volume} of a sub-manifold $W$ computed with respect to the induced metric on $W$. One can consider for each disk $W_r(x)$ the exponential volume growth rate of its iterated under the application $f$, as \[ \chi_{W}(x,r)=\limsup_{n\to\infty} \frac{1}{n}\log\Vol(f^{n}W_r(x)),\]
where we then consider the maximum volume growth rate of $\W$
\[\chi_{\W}(f)=\sup_{M\ni x} \chi_{\W}(x,r).\]
We refer to `exponential growth rate of unstable disks' when we are dealing with a partially hyperbolic diffeomorphism and it is considered the corresponding unstable foliation.

\subsection{Some tools}

We consider here some definitions and terminology to be used throughout the remaining sections.
\medskip

{\em Non-degeneracy}: A $u$-form $\omega$ on a manifold $M$ is a choice, for each $M\ni x$, of an alternating $u$-multilinear map $\omega_x\colon T_xM\times\cdots\times T_xM\to\Real$, which depends smoothly on $x$. It is denoted by $\Lambda^{u}(M)$ the set of all $u$-forms. A {\em
closed} $u$-form
$\omega\in\Lambda^{u}(M)$ is said to be {\em non-degenerate} on a tangent
subbundle $E^{u}\subset TM$ (with fiber dimension equal to $u$) if for every
$M\ni x$ and any set of linearly independent vector
$\{\ve_1,\ldots,\ve_u\}\subset E^{u}_x$ the following holds,
$\omega_x(\ve_1,\ldots,\ve_u)\neq 0$.

\medskip

{\em Cone Structure}: Let $V$ be a $d$-dimensional vector space with an inner product structure $\seq{\cdot,\cdot}$, and let $\norm{\cdot}$ be the induced norm. Consider a $u$-dimensional subspace $E\subset V$. Let $E\oplus E^{\bot}=V$ be a splitting of $V$. Given $0<\alpha\in\Real$ we define the $\alpha-cone$ with {\em core} $E$ as the following set, \[ \C^{u}_{\alpha}(E)=\set{(\ve,\ve^{\bot})\,\,| \,\,\ve\in E,\,\, \ve^{\bot}\in E^{\bot}\,\, \text{and}\,\, \norm{\ve^{\bot}}\leq\alpha\norm{\ve}}.\]

A continuous {\em cone field} on a subset $K$ of a manifold $M$ is a continuous association of cones $\set{\C_x}$ in the vector spaces $T_xM$ of tangent vectors on $M$, $x\in K$.

A useful condition for partial hyperbolicity involving cone fields is the {\em cone criterium} ~\cite[Appendix B~]{BDV05}, which establishes that partial hyperbolicity is equivalent to having cone fields on $M$ with invariance and increase property for forward iterations and backward iterations.

\medskip

{\em Lyapunov Exponents}: Let us consider a $C^{1}$-diffeomorphism $f$ of a $d$-dimensional manifold $M$ and an ergodic measure $\mu$. From Oseledet's theorem~\cite{Man83} it follows that there exist real numbers $\lambda_1<\lambda_2\cdots<\lambda_m$, $m\leq d$, called {\em Lyapunov exponents}, and a decomposition of the tangent bundle $T_xM = E_1(x)\oplus\cdots\oplus E_m(x)$ such that for every $1\leq j\leq m$ and every $0\neq\ve\in E_j(x)$ we have \[ \lim_{n\to\pm\infty} \frac{1}{n}\log\norm{D_x f^{n}\ve}=\lambda_j.\]
An $f$-invariant ergodic measure $\mu$ is said to be {\em a hyperbolic measure} if none of the Lyapunov exponents for $\mu$ are zero and there exist Lyapunov exponents with different signs.

\medskip

{\em Normally Hyperbolic Foliation and Lamination}: We recall some results about the persistence of normally hyperbolic compact laminations~\cite{HPS77}. A {\em continuous foliation}, $\F=\set{\F_x}_{M\ni x}$, on a manifold $M$ is a division of $M$ into disjoints submanifold, $\F_x$, called {\em leaves} of the foliations (the leaves have the same dimension, say $u$). Each
leaf is connected, such that each point $M\ni p$ admits a neighborhood \, $U\ni p$ and there exists a homeomorphism $\varphi\colon D^{u}\times D^{m-u}\to U$ \,(called {\em foliation box}) sending each $D^{u}\times \{y\}$ into the leaf through $\varphi(0,y)$. If the foliations boxes are $C^{1}$, $\F$ is said to be {\em $C^{1}$-foliation}. A $C^{1}$-foliation is said to be a $C^{1}$-{\em lamination} if the tangent planes of the leaves give a continuous $u$-plane subbundle of $TM$.

Consider a lamination $\El$ and let $f\colon M\to M$ be a $C^{1}$-diffeomorphism {\em preserving} lamination $\El$ (the dynamic sent each leaf into a leaf). We will say that $\El$ is {\em normally hyperbolic} if there exists a $Df$-invariant splitting, $TM=E^{u}\oplus T\El \oplus E^{s}$, of the tangent bundle, where the decomposition is partially hyperbolic.

Two laminations $\El_f$ and $\El_g$ invariant under $f$ and $g$ respectively are said to be {\em leaf conjugate} if there exists a homeomorphism $h$ such that for every $M\ni x$, $h$ carries laminae to
laminae, i.e. $h(\El_f(x))=\El_g(hx)$, and at level of leaves $h$ behaves as a usual conjugation, i.e. $h(\El_f(fx))=\El_g(g\circ hx)$.

A lamination $\El$, which is preserved by $f$, is {\em structurally stable} if there exists a neighborhood $\U\ni f$ such that each $g\in\U$ admits some $g$-invariant lamination $\El_g$, which is leaf conjugated to $\El$.

We highlight a result in the theory, concerning the condition that is required to keep intact a lamination under dynamic perturbations~\cite{HPS77}.
\begin{thmx}
 If $f$ is a $C^{1}$ diffeomorphism of $M$ which is normally hyperbolic at the $C^{1}$-foliation $\F$, then $(f,\F)$ is structurally stable.
\end{thmx}

\section{Examples}

This section is devoted to giving some examples for which the basic question is positively answered. In the examples, we basically identify two things: (1) {\em Upper bound problem}: to know whether the homological entropy is an upper bound for the topological entropy. (2) {\em Approximation problem}: the existence of hyperbolic sets with good index of hyperbolic sets, i.e. hyperbolic sets with unstable index equal to algebraic index.

\subsection{Linear Anosov map}

This is the first example to be considered, being that the basic question arises from it by observing its properties. As a prototype, it is worthwhile to keep in mind and consider the Thom-Anosov diffeomorphism.

Let $f_A$ be a hyperbolic linear automorphism induced by $A\in SL(d,\Int)$. As we know, we have a canonical, invariant (under $A$) splitting $T\Toro^{d}= E^{u}_A\oplus E^{s}_A$ of the tangent bundle, where $E^{u}_A$ is the eigenspace corresponding to eigenvalues $\lambda_i$ greater than 1 in absolute value and $E^{s}_A$ the eigenspace of the remaining eigenvalues. Suppose that $u=\dim E^{u}_A$.

Then, on one hand $h_{top}(f_A)=\sum_{|\lambda_i|>1}\log |\lambda_i|$. On the other hand, by assuming that the eigenvalues $\lambda_i$ are sorted in order of decreasing absolute value, we have the equality $sp({f_A}_{\ast,j})=\abs{\lambda_1\cdots\lambda_j}$, and thus by hypothesis, $h_H(f)=\log sp(A_{\ast,u})$.

\subsection{Anosov system with orientable bundle}

Here, we consider an Anosov diffeomorphism with orientable stable and unstable bundles. As such, one has a similar result to the Anosov lineal case.

Was proved by Ruelle--Shub--Sullivan that if $f\colon M\to M$ is an Anosov diffeomorphism such that the corresponding unstable and stable sub-bundles $E^{u}_{f}$ and $E^{s}_{f}$ are orientable, then $h_{top}(f)=h_H(f)$ and $h_H(f)=\log sp(f_{\ast,u})$. It is also shown that $u=\dim E^{u}_{f}$. To go into more details we refer the reader to~\cite{RS75}.

\subsection{Smooth examples on surfaces}

These examples, due to Katok, are no longer so trivial; the importance is that the approximation problem by hyperbolic sets is tackled entirely assuming existence of hyperbolic measures.

Let us begin by stating one of Katok's results~\cite{Katok80}. Assume that $f\colon M\to M$ is $C^{1+\alpha}$-diffeomorphism, $\alpha>0$, of a manifold $M$. Let $\mu$  be an $f$-invariant ergodic and hyperbolic measure such that $h_{\mu}(f)>0$. Then for every $\eps>0$, there exists $\Lambda\subset M$ a hyperbolic set such that $h_{top}(f\mid \Lambda)>h_{\mu}(f)-\eps$.

If $M$ is a surface, the result can be refined. In fact, by the variational principle of entropy, we know that $h_{top}(f)=\sup_{\mu} h_{\mu}(f)$. Furthermore, by Ruelle's inequality for entropy~\cite{Man83}, we have that one of the Lyapunov exponents of $\mu$ (assume that $h_{\mu}(f)>0$) is positive and the other is negative.
Therefore, for surfaces, the topological entropy of a $C^{1+\alpha}$-system can be approximated by topological entropy of hyperbolic sets of index one.

Moreover, what we need to observe is that if the induced action on homology level $f_{\ast,1}\colon H_1(M,\Real)\to H_1(M,\Real)$ is hyperbolic, we then actually have a similar approximation of $h_H(f)$ by entropy of hyperbolic sets of index one. This is due to Manning's result, cited in~\cite{Katok86}, which establishes that $h_{top}(f)\geq\log sp(f_{\ast,1})$.

\subsection{Absolutely partially hyperbolic diffeomorphisms on $\Toro^{3}$ isotopic to a linear Anosov}
This is the first example in a non-hyperbolic setting to be analyzed, which is important for our context.

Let $A\colon\Toro^{3}\to\Toro^{3}$ be a hyperbolic linear
automorphism, and consider $f\colon \Toro^{3}\to \Toro^{3}$ an absolutely partially hyperbolic diffeomorphism, such that $f$ is homotopic to $A$. Denote with $\lambda^{s}\leq \lambda^{c}\leq \lambda^{u}$ the Lyapunov exponents associated to an ergodic $f$-invariant measure, $\mu$. Then, Ures
in~\cite{Ures12} has proved the following result:
There exists $\mu$, a unique ergodic $f$-invariant entropy maximizing
measure of $f$. This measure is hyperbolic and the central Lyapunov exponent,
$\lambda^{c}$, has the same sign as the linear part, $\lambda^{c}_{A}$.

Notice that, by choosing the inverse map, $f^{-1}$, we may actually assume that the linear map has a 2-dimensional unstable direction. In the context referred, if $f\colon\Toro^{3}\to\Toro^{3}$ is $C^{1}$, then for every $\eps>0$, there exists $\Lambda\subset M$ a hyperbolic set with also 2-dimensional unstable bundle, such that $h_{top}(f\mid\Lambda) \geq h_H(f)-\eps$.

\subsection{Certain perturbations of linear Anosov diffeomorphism with a dominated splitting}
Far from partial hyperbolicity there exists a family of diffeomorphisms which are in the isotopy class of an Anosov diffeomorphism. Elements of such a family, under certain conditions, retain the existence and uniqueness property of a measure of maximum entropy in the case of central dimension larger than one~\cite{BF13}.

\section{Proof of Theorem A}
Our goal in this section is to give a full sketch of the proof of Theorem A. To prove Theorem A, we combine the existence of a closed differential $u$-form $\omega\in\Lambda^{u}(\Toro^{d})$ which is non trivial on $\W^{u}_A$, the unstable bundle of hyperbolic linear automorphism $A$ and a result, due to Saghin, relating the exponential growth rate of unstable disks.

 \begin{thm} \label{t:A} Let $f\in$ {\em $\PH^{0}_{A}$}$(\Toro^{d})$ with $\dim E^{c}_{f}=1$. Then the following holds:
   \begin{enumerate}
      \item Let $\mu$ be an ergodic $f$-invariant measure, $0<\eps_0<\abs{\lambda^{c}_A}$. Suppose that $h_{\mu}(f)>h(A)-\eps_0$ then $\mu$ is a hyperbolic measure with index $u=\dim E^{u}_A$.
      \item For every $\eps>0$ there exists a hyperbolic set $\Lambda_{\eps}\subset\Toro^{d}$ such that
          $$h_{top}(f\mid\Lambda_{\eps})\geq \log sp(f_{\ast,u})-\eps$$ where $u=\dim E^{u}_{\Lambda_{\eps}}$, and $h_H(f)=\log sp(f_{\ast,u})$.
    \end{enumerate}
 \end{thm}
While on $\Toro^{3}$, the result follows directly from the quasi-isometry
property of unstable leaves; however, in a larger dimension it follows from the
existence of a closed $u$-form which is required to be non-degenerate on the
unstable direction. So, to start dealing with the problem, the next key
proposition shows that, to have a positive definite differential form is a
non-isolated property, as well as it being a property that extends to the closure.

\begin{prop} \label{p:OpenClose}
The existence of a closed $u$-form, $\omega$, which is non-degenerate on the
unstable bundle $E^{u}$ is an open and closed condition in {\em
$\PH_{A}$}$(\Toro^{d})$.
\end{prop}

\begin{proof}
The condition of being open is trivial. In fact, take $f\in \PH_A(\Toro^{d})$
and let $\omega\in\Lambda^{u}(\Toro^{d})$ be a closed $u$-form which we
assume to be non-degenerate on $E^{u}_{f}$ bundle.
We need to show that there exists a neighborhood \, $\U$ of $f$ in
$\PH_A(\Toro^{d})$, such that for every $g\in \U$ there exists a closed
$u$-form, $\omega_g$, which is non-degenerate on $E^{u}_{g}$.\\
We have that there exists $0<\alpha\in\Real$ such that $\omega_x$ is positive
definite over a cone $\C(E^{u}_{f}(x),\alpha)$. Notice that by the
continuity of bundles, $\alpha$ can be taken locally constant. Furthermore,
by using a compactness argument, $\alpha$ does not depend on $x\in\Toro^{d}$.
On the other hand, for $g\in \PH(\Toro^{d})$ $C^{1}$-close enough to $f$,
also we have that $E^{u}_{g}(x)\subset\C(E^{u}_{f}(x),\alpha)$. So the
same $u$-form, $\omega$, works for $g$.

In order to prove the closed property, consider the $C^{1}$-convergence
$f_n\to f$ where $f_n, f\in \PH_A(\Toro^{d})$ and a closed $u$-form,
$\omega_n$, non-degenerate on $E^{u}_n$ the unstable bundle corresponding to
$f_n$. To construct a non-degenerated closed $u$-form associated with $f$, we
attempt to approach $n$ sufficiently so that the `cone axis', where the form
$\omega_n$ is positive, becomes $C^{0}$-close to the unstable bundle of $f$.
Although this can be done in a uniform way and as closely as possible, the
problem is that the cone angle may be small, and as a result not large enough
to contain the unstable bundle of $f$. However this is corrected by pushing
forward the unstable bundle of $f$ so that it fits within the cone.

To do this, take $N>0$ big enough to ensure that the $E^{cs}_{N}$ and
$E^{u}_{f}$ bundles become transversal to each other. We can do this because:
\begin{itemize}
  \item $E^{cs}_{f}$ \, is transversal to \, $E^{u}_{f}$.
  \item $E^{cs}_{n} \to E^{cs}_{f}$ \, $C^{0}$ in a uniform way.
\end{itemize}
So, there exists $0<\alpha_0\in\Real$ such that the angle
$\angle(E^{sc}_N(x),E^{u}_f(x))\geq\alpha_0$ for all $x\in\Toro^{d}$. We can
also modify $N$ in order to have $\angle(E^{u}_N(x),E^{u}_f(x))<\delta$ for
every $x\in\Toro^{d}$, where $\delta>0$ is small.\\
Associated to this $\delta>0$, there exists $m\in\Nat$ so that for every
$x\in\Toro^{d}$ and $\ve\in E^{u}_{f}(x)$ we have
$$D_xf^{m}_N(\ve)\in \C( E^{u}_{N}(f^{m}_{N}x),\beta).$$
To finish the proof of proposition we define $\omega_f\colon\Toro^{d}\to
\Lambda^{u}(\Toro^{d})$ a $u$-form given by
\[\omega_f=(f^{m}_N)_{\ast}\omega_N.\]
Notice that $\omega_f$ is indeed a closed form and by construction it is
also non-degenerate on unstable bundle $E^{u}_{f}$.
\end{proof}

\begin{rem}
Notice that if a diffeomorphism $f$ has a closed $u$-form which is
non-degenerate on unstable bundle, then any element of the connected component
that contains $f$ also has the same property.
\end{rem}

The next two theorems help us conclude the result. The first theorem considers
$f$ a $C^{1}$-partially hyperbolic diffeomorphism with one dimensional center
direction, where Hua, Saghin and Xia proved a refined version of the
Pesin-Ruelle inequality~\cite{HSX08}. And the second theorem, due to Saghin,
establishes the relation between the exponential growth rate of unstable
discs and its corresponding at homology~\cite{Saghin12}.

\begin{thm}[Hua, Saghin, Xia]\label{t:HSX}
Let $\nu$ an ergodic $f$-invariant measure and $\lambda^{c}(\nu)$ its
Lyapunov exponent corresponding to the center distribution. Then the
inequality holds $\,$ $h_{\nu}(f)\leq\lambda^{c}(\nu)+\chi_{u}(f)$.
\end{thm}

\begin{thm}[Saghin] \label{t:Saghin}
Let $f\colon M\to M$ be a $C^{1}$-partially hyperbolic diffeomorphism such
that there exists a closed $u$-form, $\omega$, which is non-degenerate on
unstable bundle $E^{u}_{f}$. Then, $\chi_{u}(f)=\log sp(f_{\ast,u})$.
\end{thm}

The proof of Theorem~\ref{t:A} is as follows:

\begin{proof}[Proof {\em (}of Theorem{\em~\ref{t:A}}{\em )}]
Now we show that the unique measure of maximal entropy is hyperbolic for each
$f\in\PH^{0}_{A}(\Toro^{d})$. In fact, we will show that the sign of the
center Lyapunov exponents, for both, $f$ and $A$, are the same. We can assume that for $A$, the Anosov diffeomorphism, the center direction is expanding; otherwise, we would use the inverse map.

For $f\in\PH^{0}_{A}(\Toro^{d})$, we consider $\mu_{f}$ the unique measure of
maximal entropy. By Theorem~\ref{t:HSX} we have:
\[h_{top}(A)=h_{top}(f)=h_{\mu_{f}}(f)\leq\lambda^{c}_{f}+\chi_{u}(f).\]
Because $\PH^{0}_{A}(\Toro^{d})\ni A$, by Proposition~\ref{p:OpenClose} and
Theorem~\ref{t:Saghin} we also have $$\chi_{u}(f)=\chi_{u}(A),$$ and since
$$h_{top}(A)=\chi_{u}(A)+\lambda^{c}(A),$$ we conclude that
$\lambda^{c}_{A}\leq \lambda^{c}_{f}$. Now, if $f$ is $C^{1}$ the existence
of hyperbolic sets with the required property follow from Katok's result~\cite{Katok80, Gel14}.
\end{proof}

\begin{rem}
Far from good regularity, i.e. $C^{1+\alpha}$ setting, the
Pesin theory fails dramatically~\cite{BCShi13}. However,
even though it is not known if Katok's theorem fails in $C^{1}$ regularity,
$C^{1+\alpha}$ assumption can be relaxed (just requiring the weaker $C^{1}$
differentiability hypothesis) when the domination condition is added~\cite{Gel14}.
\end{rem}

\section{Proof of Theorem B}

In order to prove Theorem B we appeal to a source of many examples, namely the skew-products. We start from a product of two linear Anosov system and locally modify the dynamics of one of the coordinate. The modification will have torus leafs as center fibers, on one of them the dynamics remains hyperbolic which is sufficient to achieve the required properties.

\medskip
The next two proposition are required.
\begin{prop}\label{p:HighEntropy}
    Let $A$ be a hyperbolic linear automorphism of the $2$-torus. For every $K>0$
    there exists, $f\colon \Toro^{2}\to\Toro^{2}$, a diffeomorphism isotopic to
    $A$ such that $h_{top}(f)>h_{top}(A)+K$.
\end{prop}
\begin{rem}
    Notice that, because hyperbolic sets persist under small perturbations, if $g$
    is $C^1$-close to $f$, there exists a hyperbolic invariant set $H_g$ close to
    $H$ such that $g|_{H_g}\colon H_h\to H_g$ and $f|_{H}\colon H\to H$ are
    topologically conjugates. Thus,
    $$h_{top}(g)\geq h_{top}(g|_{H_g})=h_{top}(f|_H)>h_{top}(A)+K.$$
\end{rem}

The following proposition tells us that, for hyperbolic linear maps, the
entropy `does   not look affected' when removing a fixed point of its domain,
i.e,

\begin{prop}\label{p:EntropyFixed}
    Let $B$ be a hyperbolic linear automorphism of the $2$-torus. Let $p\in\Toro^{2}$
    be a fixed point of $B$. For every $\eps>0$, there exists
    $\delta>0$, such that $h_{top}(B|_{\Lambda_{\delta}})>h_{top}(B)-\eps/2$ where
    $\Lambda_{\delta}$ the maximal invariant set contained in $\Toro^{2}\setminus
    B_{\delta}(p)$.
\end{prop}

In the next proposition we establish an invariance property of cones. To do
so, assume that $L\colon\Real^{4}\to\Real^{4}$ is a linear map written as a
matrix
$$
   L=
      \left(
        \begin{array}{cc}
          B_N & X \\
          0 & Y \\
        \end{array}
      \right)
$$
of linear maps \, $\Real^{2}\to\Real^{2}$. Where $Y$ is invertible and
$$
   B_N=
      \left(
        \begin{array}{cc}
          \lambda^{N} & 0 \\
          0 & \lambda^{-N} \\
        \end{array}
      \right),
    \ \ \ \ \Real\ni\lambda>1.
$$

Now consider $E_1=\Real e_1$, a linear subspace generated by $e_1=(1,0,0,0)$
and a cone
$$ \C^{u}(E_1,\gamma_1, \gamma_2)=\set{\ve\in\Real^{4}\colon
\abs{\ve_2}\leq\gamma_2\abs{\ve_1},\ \ \ \abs{\ve_3},
\abs{\ve_4}\leq\gamma_1\abs{\ve_1}}.$$ Notice that $E_1$ is invariant by $L$.
Similarly, we define a cone $\C^{s}(E_2,\gamma_1, \gamma_2)$, where
$e_2=(0,1,0,0)$ and $E_2=\Real e_2$ is also $L$-invariant. The basic fact we
want is,

\begin{prop}\label{p:InvarianceCones}
Let $N$ be a positive integer such that $(\norm{Y}+1)<100^{-1}\lambda^{N}$
and $\norm{Y^{-1}}<100^{-1}\lambda^{N}$. Consider $0<K\in\Real$ so that
$\norm{X}<K$. Then
\begin{enumerate}
  \item There exist $\gamma_1, \gamma_2>0$, so that
      $\C^{u}(E_1,\gamma_1, \gamma_2)$ becomes a $L$-invariant cone.
  \item There exists $\lambda>1$ such that for all $\ve\in\C^{u}$,
      $\norm{L\ve}\geq\lambda\norm{\ve}$.
\end{enumerate}
\end{prop}

\begin{proof}[Proof {\em (}of Theorem {\em B)}]
We start by considering a hyperbolic linear automorphism of the $2$-torus
$A$. Therefore by Proposition~\ref{p:HighEntropy} we have
$f_t\colon\Toro^{2}\to\Toro^{2}$ an isotopy between $f_0=A$ and $f_1=f$,
where $h_{top}(f)>h_{top}(A)+\varepsilon$. Notice that we can think of $f_t$
as being in fact a diffeotopy.

Let $\lambda>1$ be an eigenvalue corresponding to a linear hyperbolic
automorphism, call it $B$, which has a fixed point $p\in\Toro^{2}$. Consider
also $N>0$ big enough in such a way $(\|D_yf_t(y)\|+1)<100^{-1}\lambda^{N}$.
Subsequently, take $\delta>0$, given by Proposition~\ref{p:EntropyFixed}, so
that we have
$$h_{top}(B^{N}|_{\Lambda_{\delta}})>h_{top}(B^{N})-\frac{\eps}{2},\ \
\text{where} \ \ \ \Lambda_{\delta}=\Toro^{2}\setminus B_{\delta}(p).$$
Finally, we define
$F\colon\Toro^{2}\times\Toro^{2}\to\Toro^{2}\times\Toro^{2}$ by
$$F(x,y)=(B^{N}x, f_{k(x)}y), \ \ \ \ \Toro^{2}\ni x, y,$$
where the smooth function \, $k\colon\Toro^{2}\to\Real$ \, satisfies \,
$k(x)=0$ \, if \, $\|x-p\|>\delta$, \, and \, $k(x)=1$ \, if\,
$\|x-p\|\leq\delta/2$.\\
The derivative of $F$ is the matrix
    $$
      DF(x,y)=
             \left(
               \begin{array}{cc}
                  B^{N} & D_{x}(f_{_{k(x)}}y)\\
                  0 & D_{y}(f_{_{k(x)}}y) \\
               \end{array}
             \right),
    $$
and we have that $F$ is a diffeomorphism and meets the requirements of
Proposition~\ref{p:InvarianceCones}.

Hereafter, we consider an open set $\U\subset \Diff^{1}(\Toro^{4})$ such
that $\U\ni F$, any $G\in\U$ is partially hyperbolic and such
that $G$ is isotopic to $F$. Remember that partial hyperbolicity is a $C^{1}$
open condition, and we know that if $G$ is $C^{0}$-close to $F$ then $G$ and
$F$ are homotopic. First of all, $F$ is indeed homotopic to the linear
hyperbolic automorphism $(B^{N},A)$ for which we know that $h_H(F)=\log
sp(F_{\ast,2})$. Also, we have
\begin{equation*}
  \begin{split}
      h_{top}(F) &\geq h_{top}(B^{N}|_{\Lambda_{\delta}})+h_{top}(f) \\
      & > h_{top}(B^{N})-\frac{\eps}{2}+h_{top}(A)+\eps \\
      & = h_{top}(B^{N},A)+\frac{\eps}{2}.
  \end{split}
\end{equation*}
So, the third item of theorem is proved.

Now, we consider $G\in \U$. Let us denote with \,
$\lambda^{-}_{N}<0<\lambda^{+}_{N}$ \, the Lyapunov exponents of \, $B^{N}$,
and by \, $\lambda^{-}_{A}<0<\lambda^{+}_{A}$ \, associated to $A$.
Let $\mu$ be an ergodic $G$-invariant measure with $h_{\mu}(G)>h_H(G)$, and
for this measure we have the exponents \, $\lambda^{-}_{N}, \lambda^{+}_{N},
\gamma_1, \gamma_2\in\Real$. Now, by Ruelle's inequality it follows that
$$\lambda^{+}_{A}+\lambda^{+}_{N}=h_H(G)<h_{\mu}(G)\leq
\lambda^{+}_{N}+\max\{0,\gamma_1\}+\max\{0,\gamma_2\}.$$
This implies that $\mu$ is actually a
hyperbolic measure with index being same to its linear part. Thus, the second
requirement is established.

Let us now prove that if we adjust the initial set $\U$, then   any
$G\in\U$ is topologically transitive. We consider the following
$F$-invariant laminations:
\begin{itemize}
\item Lamination by leaves homeomorphic to $\Real\times\Toro^{2}$:
    $$\F^{s}=\set{\set{L^{s}_{x}}_{\Toro^{2}\ni x}\ | \ \
    L^{s}_{x}=W^{s}(x,B^{N})\times\Toro^{2}},$$ analogously  we have the
    $\F^{u}$-lamination.
\item Lamination by torus $\Toro^{2}$:
    $$\G=\set{\set{T_x}_{\Toro^{2}\ni x}\ | \ \
    T_x=\{x\}\times\Toro^{2}}.$$
\end{itemize}
These lamination are normally hyperbolic and $C^{1}$; therefore, they are plaque
expansive for $F$, and as such we have $C^{1}$-persistence of such a
laminations.

Robust transitivity is due to these properties, as well as because the
dynamic induced on the space of the leaves remains the same. Namely, consider
$\U\ni F$ the initial $C^{1}$-neighborhood and reduces it in such a way
the above laminations persist. So for any $G\in\U$ we have the following
properties,
\begin{enumerate}
    \item $G|_{\Lambda}$ is transitive on the torus $\Lambda =
        \h_{G}(\{p\}\times \Toro^{2}) \subset \Toro^{4}$,
    \item  $\bigcup_{\Lambda\ni z} W^{ss}_{G}(z)$ is dense on
        $\Toro^{4}$.
\end{enumerate}
In order to establish ($2$), by using $\h^{s}_{G}$, one can prove that:
\[\bigcup_{\Toro^{2}\ni z}W^{ss}_{F}(p,z)=W^{ss}(p,B^{N})\times\Toro^{2},\]

and so, the item $(2)$ of the theorem is proved. Note that (1) is easy to see because the restriction of $F$
to $\{p\}\times\Toro^{2}$ is equal to $A$, so by structural stability of $A$
the same holds for $G$.

To see the transitive property, take \, $U, V\subset\Toro^{4}$ \, open sets.
Take $(p, q)\in\Toro^{4}$ a periodic point of $G$, $G^{k}(p,q)=(p,q)$,
such that for some $0\leq\ell<k$, \[W^{ss}(p,q)\cap U\neq\emptyset
\,\,\,\,\,\text{and} \,\,\,\, W^{uu}(G^{\ell}(p,q))\cap V\neq\emptyset\]
Take $D\subset U$ a disk transverse to $W^{ss}(p,q)$, so
by $\lambda$-lemma, $G^{nk}(D)$ converges in compact parts, thus there
exists $n>0$ such that $G^{nk+\ell}(U)\cap V\neq\emptyset$.
\end{proof}


\section{Proof of Theorem C}

The example constructed here will be derived from an Anosov diffeomorphism. We starts with a indecomposable linear hyperbolic automorphism, we locally modify the central fiber to get enough entropy while maintaining control over the domination of bundles.

\medskip

We start by considering $f_A\colon\Toro^{4}\to\Toro^{4}$, a linear Anosov
automorphism induced in $\Toro^{4}$ by a linear map $A\colon \Real^{4} \to
\Real^{4}$ with eigenvalues $\lambda_4>1>\lambda_3>\lambda_2>\lambda_1>0$.

Let \, $T\Toro^{4} = E_{1}\oplus E_{2}\oplus E_{3}\oplus E_{4}$ \, be the
splitting associated to the eigenvalues. Take a fixed point $\Toro^{4}\ni
x_0$, a neighborhood $U\ni x_0$ and $\varphi\colon U\subset \Toro^{4}\to
D_1\subset \Real^{4}$, a chart diffeomorphism where
$D_r=\set{x\in\Real^{4}\,\,|\,\,\, \norm{x}<r} $, satisfying \[ \varphi\circ
f_A\circ \varphi^{-1}(x)=(\lambda_1 x_1,\lambda_2 x_2,\lambda_3 x_3,
\lambda_4 x_4)\] for all $x\in D_{r_0}$, where $r_0$ is so small that
$f_A\circ\varphi^{-1}(D_{r_0})\subset D_1$.\\
Take $U_0=\varphi^{-1}(D_{r_0/2})$.  We can suppose that $U\subset
B_{\delta_0}(x_0)$, where $\delta_0$ is small enough so that two different
lifts of $U_0$ are at a distance of at least $5000\delta_0$.

Because $f_A\colon\Toro^{4}\to\Toro^{4}$ is an Anosov diffeomorphism, we have
that $A$ is topologically stable in the strong sense, as proved by Walters~\cite{Wa70}.
So given $0<\eps\ll\delta_0$, there exists
$0<\delta<\delta_0$ small with the property that any diffeomorphism $f$ with
$\delta$-$C^{0}$-distance to $A$ is semiconjugated to $A$. The semiconjugacy
map $h$ is at $\eps$-$C^{0}$-distance from the identity map.

We shall modify $f_A$  inside a suitable ball such that we get a new
diffeomorphism $f_0\colon\Toro^{4}\to\Toro^{4}$. In fact, we choose $r_0$
small enough so that $U_0\subset B_{\delta/2}(x_0)$. Finally, we define the
map $f_0$ by:
\[f_0(x) = \Bigg\{
\begin{matrix} \ \ \ \ \ \ f_A(x)  \ \ \ \ \ \ \ \ \ \text{when} \ \ x\notin U_0\\
\varphi^{-1}\circ F \circ \varphi(x) \ \ \ \text{when} \ \ x\in U_0
\end{matrix}\Bigg.\]
where, if $x=(x_1,x_2), \, y=(y_1, y_2)\in \Real^{2}$, $F$ is written as
$$F(x_1,y,x_2)=(\phi(x), h(x,y), k(x)),$$
and the involved functions $\phi,\ \ h,\ \ k$, will be specified below.

What we want to do is to construct an open set \, $\U\subset
\Diff^{1}(\Toro^{4})$, $\U\ni f_0$, such that any $f\in \U$ is required to
verify the following properties:
\begin{enumerate}
    \item  $f$ is dynamical coherent.
    \item Let $\C^{u}$ the unstable cone field around the
        subspace $E_4$ which is preserved by $A$. Then, there exist
        $N\in\Nat$ and $\lambda>1$ such that $D_xf^{N}(\C^{u}(x))\subset$
        int $(\C^{u}(f^{N}x))$ and for every $\ve\in\C^{u}(x)\setminus
        \set{0}$ we have $\norm{D_xf^{N}\ve}>\lambda\norm{\ve}$.
  \item There exists a continuous and surjective map,
      $h_f$, \, $h_f\circ f=f_A\circ h_f$. The
      semiconjugacy map $h_f$, is 1-1 restricted to the unstable manifold
      $W^{u}_f$, and sends $cs$-leaves of $f$ on $s$-leaves of $A$.
 \end{enumerate}
We shall construct $f_0$, making sure that it verifies the desired properties
(1), (2), and (3), and such that by perturbing $f_0$, the properties still
remain.

First of all, we proceed detailing the construction of $f_0$; to this, we
locally modify the linear map along the central direction $E_2\oplus E_3$.
The resulting application will be denoted by $h$. The entire construction of
$h$ is summarized by saying `pasting a horseshoe'. So, let $g\colon
B_{r}\subset \Real^{2}\to B_{r}$ be a map such that

\qquad \qquad \qquad \qquad $\Omega(g)=\{0\}+H$, where $H\subset B_{\ell}$
for $\ell\leq \frac{r}{3}$,

\qquad \qquad \qquad \qquad $g(0)=0$ \ \ \ \ is an attracting fixed point,

\qquad \qquad \qquad \qquad $g(H)=H$ \ \ is conjugate to a $n$-shift.\\
Here, $B_r=\set{x\in\Real^{2}\,\,|\,\,\, \norm{x}<r}$ and $n\in\Nat$ is taken
so that $\log n>\log \lambda_1$. We now consider an isotopy $g_t\colon B_r\to
B_r$ between the maps $g_0(x_2,x_3)=(\lambda_2 x_2, \lambda_3 x_3)$ and
$g_1=g$ such that $g_t=g_0$ for all $t\in [0,1]$ and $x\notin B_{r/2}$. It is
also considered a bump-function $\chi\colon \Real\to\ [0,1]$ such that
$\chi(t) = 0$ if $\abs{t}\geq \delta_2$, $\chi(t) = 1$ if $\abs{t}\geq
\delta_1$, and $\chi(-t)=\chi(t)$ for all $\Real\ni t$. We define:
$$h(x,y)=g_{\chi(\norm{x})}(y).$$

Observe that from $g$, it is possible to construct a function $\widetilde{g}$
so that the topological entropy does not change while maintaining control on
the growth of its derivative. To do so, just take a homothety and call it
$\zeta$, and it is sufficient to consider the map $\zeta\circ g\circ
\zeta^{-1}$. The map is well glued in the complement of a suitable ball, a
region where $g$ is a linear map.

The second modification is done in the direction corresponding to the bundle
$E_{1}\oplus E_{4}$. With this, we attempt to maintain the
dominance of the bundles. We consider a linear map $T\colon
\Real^{2}\to\Real^{2}$, defined by $T(x,y)=(\lambda_1 x, \lambda_4 y)$, and
real numbers $\hat{\lambda}_1<\lambda_1,\ \hat{\lambda}_4>\lambda_4$. The
modification will lead to the maps $\phi$ and $k$, so that the application
$(\phi (x,y),k(x,y))$ has strong $(\hat{\lambda}_1,
\hat{\lambda}_4)$-contraction/expansion in a neighborhood of the origin, and
so that outside a larger ball the map becomes equal to the linear map
$T(x,y)$.

It can adapt $R(a_i, b_i)$ so that it remains in a domain as small as
required, also ensuring that $\hat{\lambda}_1, \hat{\lambda}_4$ are chosen so
that they meet the estimate: $\hat{\lambda}_1<K^{-1}$, $\hat{\lambda}_4>K$,
where $K=\norm{Dg}$.

To ensure that $f_0$, constructed as such, becomes partially hyperbolic, we
can modify $f_0$ so that it verifies a cone-criterium. For this, consider
$\eps_0>0$ and a ball $D_{\eps_0}\subset\Real^{4}$. So, there exists
$N_{\eps_0}\in\Nat$ such that
$$x\in D_{\eps_0},\ \ Fx\notin D_{\eps_0}\ \ \Longrightarrow \ \ F^{k}x\notin
D_{\eps_0}, \ \ \text{for all}\ \ 1\leq k\leq N_{\eps_0}.$$ We have
$N_{\eps_0}\to\infty$ as $\eps_0\to 0$.

Consider $\alpha>0\in\Real$ such that for all $x$,
$\angle(D_xFe_4,e_4^{\bot})\geq\alpha$ where $e_j\in\Real^{4}$ are the
canonical vectors, and $e_4^{\bot}=\langle e_1,e_2,e_3\rangle$. Let
$\theta>0$ be the cone angle, selected small enough such that
$$\angle(D_xF\ve,e_4^{\bot})\geq\frac{\alpha}{2}, \ \ \text{for all}\ \
\ve\in \mathcal{C}^{u}_{\theta}(E_4(x)).$$

Let $N\in\Nat$ big enough so that $$\angle(\mathcal{C}^{u}(E_4),
e_4^{\bot})\geq\alpha\ \ \Longrightarrow\ \ \varphi\circ f^{N}_A\circ
\varphi^{-1}\mathcal{C}^{u}(E_4)\subset \mathcal{C}_{\theta}^{u}(E_4).$$

Now, take $\eps_0>0$ small such that $N_{\eps_0}>2N$, and take an appropriate
homothetic transformation $\zeta\colon B_r\to B_{\eps_0}$. We can consider
the new map $f_0$, by considering the conjugation $\zeta\circ
F\circ\zeta^{-1}$. Of course, the condition
$\angle(D_xFe_4,e_4^{\bot})\geq\alpha$ remains for the new map.

The second condition required in the theorem is resolved in a way similar to
the previous theorem.

We need to observe that the map $f_0$ is in fact dynamically coherent. This
follows as an immediate consequence of ~\cite{FPS13}, because by construction
$f_0$ verifies the $SADC$ and $properness$ conditions (see~\cite{FPS13} for
formal definitions) which are sufficient to achieve the integrability of
$cs,cu$ bundles. It is also proved that such conditions remain valid when a
small perturbation is added.

The properties (2) and (3) are persistent in a $C^{1}$-neighborhood of
$f_0$. We just need to check that the semiconjugacy\, $h_f$\, is 1-1 when
restricted to unstable leaves; however, because the unstable manifold\,
$E^{u}_A$\, is one-dimensional, this is a consequence of a property that we
have by looking at its lift:
$\widetilde{h}_f(\widetilde{x})=\widetilde{h}_f(\widetilde{y})$\, if and only
if there exists $K>0$ such that
$\|\widetilde{f}^n(\widetilde{x})-\widetilde{f}^n(\widetilde{y})\|<K$ for
every $n\in\Int$.

Let us now prove that every $f\in\U$ is topologically transitive. We consider
two properties:
\renewcommand{\labelenumi}{\alph{enumi}$)$ }
\begin{enumerate}
  \item The existence of $L>0$ such that every unstable arc with length bigger
than $L$ intersects every $cs$-disc with internal radius $5\delta$.
  \item By the classical Ma\~{n}\'{e}-Bonatti-Viana argument the backward iterates of any $cs$-disc will contain a disc of radius bigger than
$5\delta$.
\end{enumerate}

So, the transitive property is immediate, in fact, suppose that $U$, $V$ are
open sets of $\Toro^{4}$. On one hand, we have that a forward iterated of
$U$, $f^{m_0}(U)$, contains an arc-segment of length larger than $L$, and on
the other hand there exists $n_0$ such that $f^{-n_0}(V)$ will contain a
$cs$-disc of radius bigger that $5\delta$. Thus, $f^{m_0+n_0}(U)\cap
V\neq\emptyset$.

As in~\cite[p.~190]{BV00}, condition (a)  will be satisfied if one
chooses sufficiently narrow cone fields. By the lemma below, condition (b)
also holds.

\begin{lem}\label{p:ManheBonattiViana}
    Consider $f\in\U$ and denote by $W^{cs}_{loc}(f)$ an arbitrary $cs$-disc.
    Then there exists $n_0\in\Nat$ such that $f^{-n_0}(W^{cs}_{loc}(f))$ has
    $cs$-radius larger than $5\delta$.
\end{lem}

\begin{proof}
We consider an arbitrary $cs$-disc in $\Toro^4$. Since it contains a strong
stable manifold, their negative iterates reach a size larger than $\eps$.
This implies that $h_f$, the semiconjugacy map, cannot collapse the disc and
so the negative iterates of the disc grow exponentially in diameter. The
classical Ma\~{n}\'{e}'s argument~\cite{Man78} then implies that there exists a point
$x_0$ in the disc and $n_0\in\Nat$ such that $f^{-n}(x_0)$ does not belong to
the perturbation region for any $n\geq n_0$. Since outside the perturbation
region the dynamic is hyperbolic, one obtains that the negative iterates of
the disc eventually reach the size of $5\delta$.
\end{proof}

\section{Considerations}

{\bf 1.} It turns out that there exists an interesting
question, which we have yet to mention. Notice that in principle, if there is
$u_0\in\Nat$ such that $h_H(f)= \log sp(f_{\ast,u_0})$, this $u_0$ may not
necessarily  be unique. However, in the setting of Theorem $A$, e\-very
partially hyperbolic diffeomorphism considered there associates a unique
value $u_0$ which maximizes the homological entropy. This follows immediately
from the uniqueness of the maximal entropy measure. The general case could be different, and should be studied in a deep way.

What one would expect is that the existence of hyperbolic sets with `good
index' still continues to exist under certain conditions.

For instance, consider a partially hyperbolic diffeomorphism $A\times
Id\colon \Toro^{3}\hookleftarrow$, where $A$ is a linear hyperbolic map on
$\Toro^{3}$ and $Id$ the identity map on $\mathbb{S}^{1}$. Consider $f$ a
small perturbation of $A\times Id$. We know that $f$ is also partially
hyperbolic and has center fibers $W^{c}(x)$ homeomorphic to $\mathbb{S}^{1}$,
where the dynamic induced on space of leaves, which is $\Toro^{2}$, is the
same as the linear dynamic, $A$. We have that $h_H(f)=\log sp(f_{\ast,1})$
and also $h_H(f)=\log sp(f_{\ast,2})$. In~\cite{HHTU10} is proved that if $f$
has one maximizing hyperbolic measure with index 1, then there is at least
another maximizing hyperbolic measure with index 2. Actually, in the context
of dynamically coherence partially hyperbolic diffeomorphism with compact one
dimensional center leaves, they form an open and dense subset.

\medskip

{\bf 2.} Related to Theorem C, we ask if the following can be
proved: there is not an example as in Theorem C which is absolutely partially
hyperbolic.

\medskip

{\bf 3.} Consider $f\in\PH(\Toro^{3})$ with splitting
$E_f^{cs}\oplus E_f^{u}$ isotopic to a linear Anosov
$A\colon\Real^{3}\to\Real^{3}$ with $\dim E_{A}^{s}=2$. Let $\mu$ be an
ergodic invariant measure such that $h_{\mu}(f)>h(A)$. Where is such a
measure supported? In Theorem C, we saw that the support of such a measure is
on an invariant surface. Is that always the case?

\medskip

{\bf 4.} Consider $f\in\PH^{0}_A(\Toro^{d})$. Is $f$ transitive?




\bibliographystyle{amsplain}
\bibliography{xbib}
\end{document}